\documentclass[12pt]{amsart}
\usepackage[cp852]{inputenc}
\usepackage{amssymb}
\usepackage{amsthm}
\usepackage{enumitem}
\newcommand{\R}{\mathbb R}

\newcommand{\N}{\mathbb N}
\newcommand{\rf}[1]{{\rm(\ref{#1})}}

\newtheorem{theorem}{Theorem}
\newtheorem{lemma}{Lemma}
\newtheorem{proposition}{Proposition}
\newtheorem{corollary}{Corollary}
\theoremstyle{remark}
\newtheorem{remark}{Remark}

\begin{document}
\author{Tomasz Szostok}
\address{Institute of Mathematics, University of Silesia,Bankowa 14,
 40-007 Katowice}
\email{tszostok@math.us.edu.pl}
\title[Functional inequalities]{Functional inequalities involving  numerical differentiation formulas of  order two}
\begin{abstract}
We write expressions connected with numerical differentiation formulas of order 
$2$ in the form of Stieltjes integral, then we use Ohlin lemma and Levin-Stechkin
theorem to study inequalities connected with these expressions.
In particular, we present a new proof of the inequality 
\begin{equation}
\label{Dr}
 f\left(\frac{x+y}{2}\right)\leq\frac{1}{(y-x)^2}\int_x^y\hspace{-2mm}\int_x^yf\left(\frac{s+t}{2}\right)ds\:dt
\leq\frac{1}{y-x}\int_x^yf(t)dt
\end{equation}
satisfied by every convex function $f:\R\to\R$
and we obtain extensions of 
  \rf{Dr}. 
 Then we deal with 
  nonsymmetric inequalities of a similar form. 
\end{abstract}
\maketitle
\section{Introduction}
Writing the celebrated Hermite-Hadamard inequality  
\begin{equation}
\label{HH}
f\left(\frac{x+y}{2}\right)\leq\frac{1}{y-x}\int_{x}^yf(t)dt\leq\frac{f(x)+f(y)}{2}
\end{equation}
in the form 
\begin{equation}
\label{HHF}
f\left(\frac{x+y}{2}\right)\leq\frac{F(y)-F(x)}{y-x}\leq\frac{f(x)+f(y)}{2}
\end{equation}
we can see that \rf{HH} is, in fact, an inequality involving two very simple 
quadrature operators and a very simple differentiation formula. In papers 
\cite{SzA} and \cite{Szostok} the quadrature operators occurring in \rf{HH}
were replaced by more general ones whereas in 
 \cite{OSz} the middle term from \rf{HH} was replaced by more general formulas used 
in numerical differentiation. Thus inequalities involving expressions of the form 
$$\frac{\sum_{i=1}^na_iF(\alpha_ix+\beta_iy)}{y-x}$$
where $\sum_{i=1}^na_i=0,$ $\alpha_i+\beta_i=1$ and $F'=f$
were considered. In the current paper we deal with inequalities 
for expressions of the form 
\begin{equation}
\label{df}
\frac{\sum_{i=1}^na_i\Phi(\alpha_ix+\beta_iy)}{(y-x)^2}
\end{equation}
(where $\Phi''=f$)which are used to approximate the second order derivative of $F$ and, surprisingly,
we discover a connection between our approach and the inequality \rf{Dr}.

First we make the following simple observation.
\begin{remark}
\label{01}
Let $f,F,\Phi:[x,y]\to\R$ be such that $\Phi'=F,F'=f,$  let $n_i,m_i\in\N\cup\{0\},i=1,2,3$ $a_{i,j}\in\R,\alpha_{i,j},\beta_{i,j}\in[0,1],
\alpha_{i,j}+\beta_{i.j}=1,i=1,2,3;j=
1,\dots,n_i$
$b_{i,j}\in\R,\gamma_{i,j},\delta_{i,j}\in[0,1],\gamma_{i,j}+\delta_{i,j}=1,i=1,2,3;j=1,\dots,m_i.$
If the inequality
\begin{multline}
\label{01e}
\sum_{i=1}^{n_1}a_{1,i}f(\alpha_{1,i}x+\beta_{1,i}y)+\frac{\sum_{i=1}^{n_2}a_{2,i}F(\alpha_{2,i}x+\beta_{2,i}y)}{y-x}\\+\frac{\sum_{i=1}^{n_3}a_{3,i}\Phi(\alpha_{3,i}x+\beta_{3,i}y)}{(y-x)^2}
\leq\sum_{i=1}^{m_1}b_{1,i}f(\gamma_{1,i}x+\delta_{1,i}y)\\+\frac{\sum_{i=1}^{m_2}b_{2,i}F(\gamma_{2,i}x+\delta_{2,i}y)}{y-x}+\frac{\sum_{i=1}^{m_3}b_{3,i}\Phi(\gamma_{3,i}x+\delta_{3,i}y)}{(y-x)^2}
\end{multline}
is satisfied for for $x=0,y=1$ and for all continuous and convex functions $f:[0,1]\to\R$ then it is satisfied  
for all $x,y\in\R,x<y$ and for each continuous and convex function $f:[x,y]\to\R.$
To see this it is enough to observe that  expressions from \rf{01e} remain unchanged if we replace $f:[x,y]\to\R$ 
by $\varphi:[0,1]\to\R$ given by $\varphi(t):=f\left(x+t(y-x)\right).$
\end{remark}

The simplest expression used to approximate the second order derivative of $f$ is of the form 
$$f''\left(\frac{x+y}{2}\right)\approx\frac{f(x)-2f\left(\frac{x+y}{2}\right)+f(y)}{\left(\frac{y-x}{2}\right)^2}$$
\begin{remark}
\label{numan}
From numerical analysis it is known that 
$$f''\left(\frac{x+y}{2}\right)=\frac{f(x)-2f\left(\frac{x+y}{2}\right)+f(y)}{\left(\frac{y-x}{2}\right)^2}-
\frac{\left(\frac{y-x}{2}\right)^2}{12}f^{(4)}(\xi).$$
This means that for convex $g$ and for $G$ such that $G''=g$ we have
$$g\left(\frac{x+y}{2}\right)\leq\frac{G(x)-2G\left(\frac{x+y}{2}\right)+G(y)}{\left(\frac{y-x}{2}\right)^2}.$$
\end{remark}
In this paper we shall obtain some inequalities for convex functions which do not follow from 
numerical differentiation results.

In order to get such  results we shall use Stieltjes integral. In paper \cite{Rajba} it was observed that the classical 
Hermite-Hadamard inequality \rf{HH}
easily follows from the following Ohlin lemma
\begin{lemma} (Ohlin \cite{Ohlin})
Let $X_1,X_2$ be two random variables such that $EX_1=EX_2$  and let  $F_1,F_2$ be their distribution functions.
If $F_1,F_2$ satisfy for some $x_0$ the following inequalities
$$F_1(x)\leq F_2(x)\;\textrm{if}\;x<x_0\;\;\textrm{and}\;\;F_1(x)\geq F_2(x)\;{ if}\;x>x_0 $$
then 
\begin{equation}
\label{m}
Ef(X_1)\leq Ef(X_2)
\end{equation}
 for all continuous and convex functions $f:\R\to\R.$
\end{lemma}
Ohlin lemma was used also in paper \cite{SzA}.
However  in the present approach (similarly as in \cite{OSz} and \cite{Szostok}) we are going to use a
more general result from \cite{LS}, (see also \cite{NP}
Theorem 4.2.7). In this theorem we use the notations from \cite{NP}.

\begin{theorem}(Levin, Stechkin)
\label{LS}
Let $F_1,F_2:[a,b]\to\R$ be two functions with bounded variation such that $F_1(a)=F_2(a).$
Then, in order that 
$$\int_a^b f(x)dF_1(x)\leq \int_a^b f(x)dF_2(x)$$
for all continuous and convex functions $f:[a,b]\to\R$
it is necessary and sufficient that $F_1$ and $F_2$ verify the following three conditions:
\begin{equation}F_1(b)=F_2(b),\end{equation}
\begin{equation}\int_a^x F_1(t)dt\leq\int_a^x F_2(t)dt,\;x\in(a,b),
\label{FGi}
\end{equation}
and
\begin{equation}\int_a^b F_1(t)dt=\int_a^b F_2(t)dt.
\label{FGe}
\end{equation}
\end{theorem}
\begin{remark}
As it easy to see, if measures occurring in Ohlin lemma are concentrated on the interval $[x,y]$ then this lemma
is an easy consequence of Theorem \ref{LS}. However Theorem \ref{LS} is more general for two reasons: it allows 
functions $F_1,F_2$ to have more crossing points than one and functions $F_1,F_2$ do not have to be 
cumulative distribution functions. Therefore we shall use this theorem even if functions $F_1,F_2$ 
have exactly one crossing point.
\end{remark}
Let now $f:[x,y]\to\R$ be any function and let $F,\Phi:[x,y]\to\R$ be such that $F'=f$ and $\Phi'=F.$
We need to write the expression 
\begin{equation}
\label{P}
\frac{\Phi(x)-2\Phi\left(\frac{x+y}{2}\right)+\Phi(y)}{\left(\frac{y-x}{2}\right)^2}
\end{equation}
in the form $$\int_x^yfdF_1$$ for some $F_1.$
In the next proposition we show that it is possible -- here for the sake of simplicity we shall work on the interval $[0,1].$
\begin{proposition}
\label{4S}
Let $f:[0,1]\to\R$ be any function and let $\Phi:[0,1]\to\R$ be such that $\Phi''=f.$ Then we have 
$$4\left(\Phi(0)-2\Phi\left(\frac{1}{2}\right)+\Phi(1)\right)=\int_x^yfdF_1$$
where $F_1:[0,1]\to\R$ is given by
\begin{equation}
\label{F3}
F_1(t):=\left\{\begin{array}{ll}
2x^2&x\leq\frac12,\\
-2x^2+4x-1&x>\frac12.
\end{array}\right.
\end{equation}
\end{proposition}
\begin{proof}
Let $F:[0,1]\to\R$ be such that $\Phi'=F.$
Now, to prove this proposition it is enough to do the following calculations
$$\int_0^1fdF_1=\int_{0}^\frac124xf(x)dx+\int_\frac12^1(-4x+4)f(x)dx$$
$$=2F\left(\frac12\right)-0\cdot F(0)-\int_{0}^\frac12 4F(x)dx-0\cdot F(1)
-2F\left(\frac12\right)+\int_\frac12^1 4F(x)dx$$
$$=4\Phi(1)-8\Phi\left(\frac12\right)+4\Phi(1).$$
\end{proof}

\begin{remark}
Observe that if $\Phi$ and $f$ are such as in Proposition \ref{4S} then the following equality is satisfied
$$\frac{\Phi(x)-2\Phi\left(\frac{x+y}{2}\right)+\Phi(y)}{\left(\frac{y-x}{2}\right)^2}=\frac{1}{(y-x)^2}
\int_x^y\hspace{-2mm}\int_x^y
f\left(\frac{s+t}{2}\right)ds\:dt.$$
\end{remark}
After this observation it turns out that inequalities involving the expression \rf{P} were considered 
in the paper of Dragomir \cite{Dragomir} where (among others) the following inequalities were obtained
\begin{equation}
\label{D}
 f\left(\frac{x+y}{2}\right)\leq\frac{1}{(y-x)^2}\int_x^y\hspace{-2mm}\int_x^yf\left(\frac{s+t}{2}\right)ds\:dt
\leq\frac{1}{y-x}\int_x^yf(t)dt.
\end{equation}
As we already know (Remark \ref{numan}) the first one of the above inequalities may be obtained using the numerical analysis
results. 

Now the inequalities from the Dragomir's paper easily follow from Ohlin lemma but 
there are many possibilities of generalizations and modifications of inequalities 
\rf{D}. These generalizations will be discussed in the following chapters.
\section{The symmetric case}
We start with the following remark. 
\begin{remark}
\label{rma}
Let $F_*(t)=at^2+bt+c$ for some $a,b,c\in\R, a\neq 0.$
It is impossible to obtain inequalities involving
$\int_x^yfdF_*$ and any of the expressions: $\frac{1}{y-x}\int_x^yf(t)dt,f\left(\frac{x+y}{2}\right),\frac{f(x)+f(y)}{2}$
which were satisfied for all convex functions $f:[x,y]\to\R.$ Indeed, suppose for example that we have 
$$\int_x^yfdF_*\leq\frac{1}{y-x}\int_x^yf(t)dt$$
for all convex $f:[x,y]\to\R.$
Without loss of generality we may assume that $F_*(x)=0,$ then from Theorem \ref{LS} we have $F_*(y)=1$
Also from Theorem \ref{LS} we get
$$\int_x^yF_*(t)dt=\int_x^yF_0dt$$
where $F_0(t)=\frac{t-x}{y-x},t\in[x,y]$
which is impossible, since $F_*$ is either strictly convex or concave.
 \end{remark}
This remark means that in order to get some new inequalities of the Hermite-Hadamard type we have 
to integrate with respect to functions constructed with use of (at least) two quadratic functions,
as it was the case in Proposition \ref{Dr}
Now we may present the main result of this section.
\begin{theorem}
\label{mainth}
Let $x,y$ be some real numbers such that $x<y$ and let $a\in\R.$
Let $f,F,\Phi:[x,y]\to\R$ be any functions such that $F'=f$ and $\Phi'=F$ and let $T_af(x,y)$ be defined
by the formula 
$$T_af(x,y)=\left(1-\frac{a}{2}\right)\frac{F(y)-F(x)}{y-x}+2a\frac{\Phi(x)-2\Phi\left(\frac{x+y}{2}\right)+\Phi(x)}{(y-x)^2}.$$
Then the following inequalities hold for all convex functions $f:$

if $a\geq 0$ then
\begin{equation}
\label{ai}
T_af(x,y)\leq\frac{1}{y-x}\int_x^yf(t)dt,
\end{equation}   
if $a\leq 0$ then
\begin{equation}
\label{aii}
T_af(x,y)\geq\frac{1}{y-x}\int_x^yf(t)dt,
\end{equation}   
if $a\leq 2$ then
\begin{equation}
\label{aiii}
f\left(\frac{x+y}{2}\right)\leq T_af(x,y),
\end{equation}   
if $a\geq 6$ then
\begin{equation}
\label{aiv}
T_af(x,y)\leq f\left(\frac{x+y}{2}\right),
\end{equation}   
if $a\geq -6$ then
\begin{equation}
\label{av}
T_af(x,y)\leq\frac{f(x)+f(y)}{2},
\end{equation}   
Furthermore:

 if $a\in(2,6)$ then the expressions $T_af(x,y),$ $f\left(\frac{x+y}{2}\right)$ 
are not comparable in the class of convex functions 

 if $a<-6$ then
expressions $T_af(x,y),$ $\frac{f(x)+f(y)}{2}$ are not comparable in the class of convex functions.
\end{theorem}
\begin{proof}
In view of Remark \ref{01} we may restrict ourselves to the case $x=0,y=1.$
Take $a\in\R,$ let $f:[0,1]:\to\R$ be any convex function and let $F,\Phi:[0,1]\to\R$ be such that $F'=f,\Phi'=F.$  
Define  $F_1:[0,1]\to\R$ by the formula
\begin{equation}
\label{F1t}
F_1(t):=\left\{\begin{array}{ll}
at^2+\left(1-\frac{a}{2}\right)t&t<\frac12,\\
-at^2+\left(1+\frac{3a}{2}\right)t-\frac{a}{2}&t\geq\frac12.
\end{array}\right.
\end{equation}
First we shall prove that $T_af(0,1)=\int_0^1fdF_1.$
Indeed, we have 
$$\int_0^1fdF_1=\int_0^\frac12fdF_1+\int_\frac12^1fdF_1=\int_0^\frac12f(t)\left(2at+1-\frac{a}{2}\right)dt$$$$+\int_\frac12^1
f(t)\left(-2at+1+\frac{3a}{2}\right)dt=F \left(\frac12\right)\left(\frac{a}{2}+1\right)-F \left(0\right)\left(1-\frac{a}{2}\right)$$$$-\int_0^\frac12F(t)2adt
+F(1)\left(1-\frac{a}{2}\right)-F \left(\frac12\right)\left(\frac{a}{2}+1\right)+\int_\frac12^1F(t)2adt$$
$$=\left(1-\frac{a}{2}\right)(F(1)-F(0))+2a\biggl(\Phi(0)-2\Phi\left(\frac12\right)+\Phi(1)\biggr)$$
Now let $F_2(t)=t,t\in[0,1].$
Then functions $F_1,F_2$ have exactly one crossing point (at $\frac12$)
and 
$$\int_0^1F_1(t)dt=\frac12=\int_0^1tdt.$$
Moreover if $a>0$ then function $F_1$ is convex on the interval $(0,\frac12)$ and 
concave on $(\frac12,1).$ 
Therefore it follows from Ohlin lemma that
for $a>0$ we have
\begin{equation}
\label{a}
\int_0^1fdF_1\leq\int_0^1fdF_2
\end{equation}
which,in view of Remark \ref{01}, yields \rf{ai} and for $a<0$ the opposite inequality is satisfied which gives \rf{aii}.

Take
\begin{equation}
\label{F3t}
F_3(t):=\left\{\begin{array}{ll}
0&t\leq\frac12,\\
1&t>\frac12.
\end{array}\right.
\end{equation}
 
It is easy to check that for $a\leq 2$ we have $F_1(t)\geq F_3(t)$ for $t\in\left[0,\frac 12\right],$
and $F_1(t)\leq F_3(t)$ for $t\in\left[\frac 12,1\right]$ and this means that from 
Ohlin lemma we get \rf{aiii}.

Suppose that $a>2.$ Then there are three crossing points of functions $F_1$ and $F_3$ $:x_0,\frac12,x_1,$
where $x_0\in(0,\frac12), x_1\in(\frac12,1)$ Function 
$$\varphi(s):=\int_0^sF_3(t)-F_1(t)dt,\;s\in[0,1]$$
is increasing on intervals $[0,x_0],[\frac12,x_1]$ and decreasing on $[x_0,\frac12]$ and on $[x_1,1].$ 
This means that $\varphi$ takes its absolute minimum at $\frac12.$
As it is easy to calculate $\varphi\left(\frac12\right)\geq 0$ if $a\geq 6$ which, in view of  Theorem \ref{LS},
gives us \rf{aiv}.

To see that for $a\in(2,6)$ expressions $T_af(x,y)$ and $f\left(\frac{x+y}{2}\right)$ are not comparable in the class 
of convex functions it is enough to observe that in this case $\varphi(x_0)>0$ and $\varphi\left(\frac12\right)<0.$

Now let 
\begin{equation}
\label{F4t}
F_4(t):=\left\{\begin{array}{ll}
0&t=0,\\
\frac12&t\in(0,1),\\
1&t=1.
\end{array}\right.
\end{equation}
Similarly as before, if $a\geq-2$ then we have  $F_1(t)\geq F_4(t)$ for $t\in\left[0,\frac 12\right]$
$F_1(t)\leq F_4(t)$ for $t\in\left[\frac 12,1\right]$ i.e. there is only one crossing point of these functions and \rf{av} is obvious.
However, for $a\in(-2,-6]$ we have
\begin{equation}
\label{14}
\int_0^\frac12 F_1(t)dt\leq\frac14=\int_0^\frac12F_4(t)dt
\end{equation}
and therefore, in view of Theorem \ref{LS} we still have 
 \rf{av}. In the case $a<-6$ inequality \rf{14} is no longer true which means that expressions
$T_af(x,y)$ and $\frac{f(x)+f(y)}{2}$ are not comparable in the class of convex functions.
\end{proof}
This theorem provides us with a full description of inequalities which may be obtained using  Stieltjes integral
with respect to a function of the form \rf{F1t}. Some of the obtained inequalities are already known. For example from 
\rf{ai} and \rf{aii} we obtain the inequality 
$$\frac{1}{(y-x)^2}\int_x^y\hspace{-2mm}\int_x^yf\left(\frac{s+t}{2}\right)ds\:dt
\leq\frac{1}{y-x}\int_x^yf(t)dt$$
whereas from \rf{aiii} for $a=2$ we get the inequality 
$$f\left(\frac{x+y}{2}\right)\leq\frac{1}{(y-x)^2}\int_x^y\hspace{-2mm}\int_x^yf\left(\frac{s+t}{2}\right)ds\:dt.$$

 However inequalities obtained  for "critical" values of $a$
i.e. $-6,6.$ are here particularly interesting. 
 In the following corollary we explicitly write these inequalities.
\begin{corollary}
\label{wn}
For every convex function $f:[x,y]\to\R$ the following inequalities are satisfied

\begin{equation}
\label{i}
3\frac{1}{(y-x)^2}\int_x^y\int_x^yf\left(\frac{s+t}{2}\right)dsdt
\leq \frac{2}{y-x}\int_x^yf(t)dt+f\left(\frac{x+y}{2}\right)
\end{equation}
and
\begin{equation}
\label{ii}
\frac{4}{y-x}\int_x^yf(t)dt\leq 3\frac{1}{(y-x)^2}\int_x^y\int_x^yf\left(\frac{s+t}{2}\right)dsdt+\frac{f(x)+f(y)}{2}
\end{equation}
 \end{corollary}
\begin{remark}
In the paper \cite{DG} S.S. Dragomir and I. Gomm obtained the following inequality 
\begin{equation}
\label{drgo}
3\int_x^yf(t)dt\leq 2\frac{1}{(y-x)^2}\int_x^y\int_x^yf\left(\frac{s+t}{2}\right)dsdt+\frac{f(x)+f(y)}{2}.
\end{equation}
Inequality \rf{ii} from Corollary \ref{wn} is stronger than \rf{drgo}. Moreover, as it was observed in Theorem \ref{mainth}
inequalities \rf{i} and \rf{ii} cannot be improved i.e. the inequality
$$\frac{1}{y-x}\int_x^yf(t)dt\leq \lambda\frac{1}{(y-x)^2}\int_x^y\int_x^yf\left(\frac{s+t}{2}\right)dsdt+(1-\lambda)\frac{f(x)+f(y)}{2}$$
for $\lambda>\frac34$ is not satisfied by every convex function $f:[x,y]\to\R$
and the inequality 
$$\frac{1}{(y-x)^2}\int_x^y\int_x^yf\left(\frac{s+t}{2}\right)dsdt
\leq \gamma\frac{1}{y-x}\int_x^yf(t)dt+(1-\gamma)f\left(\frac{x+y}{2}\right)
$$
with $\gamma<\frac23$ is not true for all convex functions  $f:[x,y]\to\R.$ 
\end{remark} 
In Corollary \ref{wn} we obtained inequalities for the triples:
$$\frac{1}{(y-x)^2}\int_x^y\hspace{-1.5mm}\int_x^yf\left(\frac{s+t}{2}\right)dsdt,\int_x^yf(t)dt,\frac{f(x)+f(y)}{2}$$
and
$$\frac{1}{(y-x)^2}\int_x^y\hspace{-1.5mm}\int_x^yf\left(\frac{s+t}{2}\right)dsdt,\int_x^yf(t)dt,f\left(\frac{x+y}{2}\right).$$
In the next remark we present an analogous result for expressions
$$\frac{1}{(y-x)^2}\int_x^y\hspace{-1.5mm}\int_x^yf\left(\frac{s+t}{2}\right)dsdt,\frac{f(x)+f(y)}{2},f\left(\frac{x+y}{2}\right).$$
\begin{remark}
Using functions: $F_1$ defined by \rf{F3} and $F_5$ given by 
\begin{equation}
\label{F5t}
F_5(t):=\left\{\begin{array}{ll}
0&t=0,\\
\frac16&t\in\left(0,\frac12\right)\\
\frac56&t\in\left[\frac12,1\right)\\
1&t=1,
\end{array}\right.
\end{equation}
we can see that 
$$\frac16f(x)+\frac23f\left(\frac{x+y}{2}\right)+\frac16f(y)\geq\frac{1}{(y-x)^2}\int_x^y\hspace{-1.5mm}\int_x^yf\left(\frac{s+t}{2}\right)dsdt$$
for all convex functions $f:[x,y]\to\R.$

Moreover it is easy to see that 
 the above inequality cannot be strengthened which  means that the inequality 
$$af(x)+bf\left(\frac{x+y}{2}\right)+af(y)\geq\frac{1}{(y-x)^2}\int_x^y\hspace{-1.5mm}\int_x^yf\left(\frac{s+t}{2}\right)dsdt$$
where $a,b\geq 0, 2a+b=1$ is not satisfied by all convex functions $f$ if $a<\frac16.$
\end{remark}

\section{The non-symmetric case}
In this part of the paper we shall obtain inequalities for  $f(\alpha x+(1-\alpha)y)$ and for $\alpha f(x)+(1-\alpha) f(y)$ 
where $\alpha$ is not necessarily equal to $\frac12.$

Now, in contrast to the symmetric case (Remark \ref{rma}), it is possible to prove inequalities using just one quadratic function
but before we do this we shall present a nonsymmetric version of Hermite-Hadamard inequality 
involving only the primitive function of $f.$
\begin{proposition}
\label{HHal}
Let $x,y$ be some real numbers such that $x<y$ and let $\alpha\in[0,1].$ Let  $f:[x,y]\to\R,$
be a convex function, let 
 $F:[x,y]\to\R$ be such that $F'=f.$ 
If $S^1_\alpha f(x,y)$ is defined by 
$$S^1_\alpha f(x,y):=\frac{\frac{-\alpha}{1-\alpha}F(x)+\frac{2\alpha-1}{\alpha(1-\alpha)}F(\alpha x+(1-\alpha)y)+
\frac{1-\alpha}{\alpha}F(y)}{y-x}$$
then the following inequality is satisfied:
\begin{equation}
\label{a0}
f(\alpha x+(1-\alpha)y)\leq S^1_\alpha f(x,y)\leq\alpha f(x)+(1-\alpha) f(y).
\end{equation}   
\end{proposition} 
\begin{proof}
As usually, the proof will be done on the interval $[0,1].$ 
Define functions $F_6,F_7,F_8:[0,1]\to\R$ by the following formulas:
\begin{equation}
\label{F6t}
F_6(t):=\left\{\begin{array}{ll}
0&t\leq 1-\alpha,\\
1&t>1-\alpha,
\end{array}\right.,
\end{equation}
\begin{equation}
\label{F7t}
F_7(t):=\left\{\begin{array}{ll}
0&t=0,\\
\alpha&t\in(0,1)\\
1&t=1,
\end{array}\right.,
\end{equation}
and 
\begin{equation}
\label{F8t}
F_8(t):=\left\{\begin{array}{ll}
\frac{\alpha}{1-\alpha}t&t\in[0,1-\alpha),\\
\frac{1-\alpha}{\alpha}t+\frac{2\alpha-1}{\alpha}&t\in[1-\alpha,1].
\end{array}\right.,
\end{equation}
We have:
$$\int_0^1F_6(t)dt=\int_0^1F_7(t)dt=\int_0^1F_8(t)dt=\alpha,$$
$$\int_0^1fdF_6=f(1-\alpha)$$
$$\int_0^1fdF_7=\alpha f(0)+(1-\alpha)f(1)$$
$$\int_0^1fdF_8=S^1_\alpha f(0,1).$$
Moreover both of the pairs $(F_6,F_8)$ and $(F_8,F_7)$ has only one crossing point.
Thus it suffices to use Theorem \ref{LS} to obtain inequalities \rf{a0}.
\end{proof}
 
\begin{theorem}
\label{NS}
Let $x,y$ be some real numbers such that $x<y$ and let $\alpha\in[0,1].$ Let  $f:[x,y]\to\R,$
be a convex functions, let 
 $F$ be such that $F'=f$ and let $\Phi$ satisfy $\Phi'=F.$  
If $S^2_\alpha f(x,y)$ is defined by 
$$S^2_\alpha f(x,y):=\frac{(4-6\alpha)F(y)+(2-6\alpha)F(x)}{y-x}-\frac{(6-12\alpha)(\Phi(y)-\Phi(x))}{(y-x)^2}$$
then the following conditions hold true:
\begin{equation}
\label{al}
S^2_\alpha f(x,y)\leq\alpha f(x)+(1-\alpha) f(y),
\end{equation}   

if $\alpha\in\left[\frac13,\frac23\right]$ then
\begin{equation}
\label{al1}
S^2_\alpha f(x,y)\geq f(\alpha x+(1-\alpha)y),
\end{equation}    
if $\alpha\in[0,1]\setminus\left[\frac13,\frac23\right]$ then expressions $S^2_\alpha f(x,y)$ and $f(\alpha x+(1-\alpha)y)$ are incomparable in 
the class of convex functions,

if $\alpha\in\left(0,\frac13\right]\cup\left[\frac23,1\right)$ then
\begin{equation}
\label{S1S2}
S^2_\alpha f(x,y)\leq S^1_\alpha f(x,y)
\end{equation}    
and if $\alpha\in\left(\frac13,\frac12\right)\cup\left(\frac12,\frac23\right)$ then $S^1_\alpha f(x,y)$ and $S^2_\alpha f(x,y)$ are incomparable in 
the class of convex functions.
\end{theorem}
\begin{proof}
Take 
$$F_9(t)=(3-6\alpha)t^2+(6\alpha-2)t,\;t\in[0,1]$$
and let $F_6,F_7,F_8$ be defined so as in Proposition \ref{HHal}.
Then we have 
\begin{multline}
\int_0^1fdF_1=\int_0^1\bigl((6-12\alpha)t+6\alpha-2\bigr)f(t)dt\\
=F(1)(6-12\alpha)-\int_0^1(6-12\alpha)F(t)dt+(6\alpha-2)\bigl(F(1)-F(0)\bigr)\\
=(4-6\alpha)F(1)+(2-6\alpha)F(0)-(6-12\alpha)\bigl(\Phi(1)-\Phi(0)\bigr)
\end{multline}
and 
$$\int_0^1F_9(t)dt=\int_0^1F_8(t)dt=\int_0^1F_7(t)dt=\int_0^1F_6(t)dt.$$
It is easy to see that functions $F_9,F_7$ have exactly one crossing point thus we have
$$\int_0^1fdF_9\leq\int_0^1fdF_7$$
which gives us \rf{al}.

Now, assume that $\alpha\in\left[\frac13,\frac23\right]$ then function $F_9$ is increasing and, consequently,
$$F_9(t)\geq F_6(t),t\in[0,1-\alpha],\;\;\textrm{and}\;\;F_9(t)\leq F_6(t),t\in(1-\alpha,1].$$
Thus  for every convex function $f$ we have
$$\int_0^1fdF_6\leq\int_0^1fdF_9$$
which yields \rf{al1}.

 Let now $\alpha<\frac13$
Then function $F_9$ is decreasing on some interval $[0,d]$ and increasing on $[d,1].$ Observe that from
the equality, 
$$\int_0^1F_9(t)dt=\int_0^1F_6(t)dt$$
we know that functions $F_9,F_6$ must have a crossing point in the interval $(0,1-\alpha),$
further these functions cross also at the point $1-\alpha.$ Thus there are two crossing points of 
$F_9,F_6$ in view of Lemma 2 from \cite{OSz} this means that expressions 
$$\int_0^1fdF_6,\int_0^1fdF_9$$
are incomparable in the class of convex functions (as claimed). The reasoning
in the case $\alpha>\frac23$ is similar.

Now we shall prove the inequality \rf{S1S2}. If $\alpha\in\left(0,\frac13\right]\cup[\frac23,1)$ then
$F_9'(0)\leq\frac{\alpha}{1-\alpha}$ and $F_9'(1)\leq\frac{1-\alpha}{\alpha}$ 
this means that functions $F_9$ and $F_8$ have only one crossing point and, therefore we have 
\rf{S1S2}.

If on the other hand $\alpha\in(\frac13,\frac12)$ then $F_9'(0)<\frac{\alpha}{1-\alpha}$ and $F_9'(1)>\frac{1-\alpha}{\alpha}$ 
and, consequently functions $F_9,F_8$ have two crossing points. Similarly as before from Lemma 2, \cite{OSz} we know 
that for $\alpha\in\left(\frac13,\frac12\right),$ $S^1_\alpha f(x,y)$ and $S^2_\alpha f(x,y)$ are incomparable in 
the class of convex functions, as claimed. It is easy to see that in the case $\alpha\in\left(\frac12,\frac23\right)$
functions $F_9,F_8$ have again two crossing points which finishes the proof.
\end{proof}

\section{Concluding remarks and examples}

In the previous sections we made an exhaustive study of two types of inequalities. Now we briefly describe the possible extensions of our results.

\begin{remark}
In order to obtain inequalities involving expressions of the form
$\frac{a_1\Phi(x)+a_2\Phi(\alpha x+(1-\alpha)y)+a_3\Phi(y)}{(y-x)^2}$  
 functions of the form
$$F_1(t):=\left\{\begin{array}{ll}
ax^2+(1-\alpha)x&t\in[0,\alpha)\\
cx^2+(1-c\alpha-c)x+c\alpha&t\in[\alpha,1]
\end{array}\right.
$$
where $c=\left(-\frac{\alpha}{1-\alpha}\right)^3$ must be used. Since the description of all 
possible cases in Theorem \ref{mainth} was already quite complicated, we shall not present 
these inequalities in details here.
\end{remark}

\begin{remark}
It is possible to use methods developed in this paper to get inequalities involving longer 
expressions of the form \rf{df}. In order to do that it is necessary to use more than two quadratic functions. For example
considering function 
\begin{equation}
F_1(t):=\left\{\begin{array}{ll}
4t^2&t\leq\frac14,\\
-4t^2+4t-\frac12&t\in\left(\frac14,\frac12\right]\\
4t^2-4t+\frac{3}{2}&t\in\left(\frac12,\frac34\right]\\
-4t^2+8t-3&t>\frac34
\end{array}\right.
\end{equation}
and using Levin-Stechkin theorem, we get the following inequality
\begin{multline*}
\frac{8\Phi(x)-16\Phi\left(\frac{3x+y}{4}\right)+16\Phi\left(\frac{x+y}{2}\right)
-16\Phi\left(\frac{x+3y}{4}\right)+8\Phi(y)}{(y-x)^2}\\
\leq\frac{1}{y-x}\int_x^yf(t)dt
\end{multline*}
where $f:[x,y]\to\R$ is any convex function and $\Phi''=f.$
\end{remark}

\begin{remark}
We have $$\int_0^1 t^2dF_9(t)=\frac56-\alpha$$ and $$\int_0^1t^2dF_6(t)=(1-\alpha)^2.$$
This means that for two values of $\alpha:$ $\frac{3-\sqrt{3}}{6}$ and $\frac{3+\sqrt{3}}{6}$
we have $\int_0^1 t^2dF_9(t)=\int_0^1 t^2dF_6(t)$ Moreover, as it was mentioned in the proof of 
Theorem \ref{NS}, functions $F_9,F_6$ have in this case two crossing points. This means that, from \cite{DLS} (see also \cite{Rajba})
we get that the inequalities:
$$S^2_{\frac{3-\sqrt{3}}{6}}(x,y)\geq f\left(\frac{3-\sqrt{3}}{6}x+\frac{3+\sqrt{3}}{6}y\right)$$
and
$$S^2_{\frac{3+\sqrt{3}}{6}}(x,y)\leq f\left(\frac{3+\sqrt{3}}{6}x+\frac{3-\sqrt{3}}{6}y\right)$$
are satisfied by all $2-$convex functions $f:[x,y]\to\R$
\end{remark}

\begin{remark} It is easy to see that all inequalities obtained in this paper in fact characterize 
convex functions (or $2-$convex functions). This is a consequence of results contained in paper \cite{BesPal}.  
\end{remark}

\end{document}